\documentclass[a4paper, 12pt]{article}
\usepackage{amsfonts}
\usepackage{mathrsfs}
\usepackage{amsmath}

\usepackage{amsthm}
\usepackage{array}
\usepackage{cases}
\makeatletter
\def\th@plain{%
  \upshape %\itshape % body font
}
\makeatother

\makeatletter
\renewenvironment{proof}[1][\proofname]{\par
  \pushQED{\qed}%
  \normalfont \topsep6\p@\@plus6\p@\relax
  \trivlist
  \item[\hskip\labelsep
        \bfseries
    #1\@addpunct{.}]\ignorespaces
}{%
  \popQED\endtrivlist\@endpefalse
}
\makeatother

\newtheorem{thm}{Theorem}%[section]

\newtheorem{lem}[thm]{Lemma}

\numberwithin{equation}{section}

\usepackage[varg]{txfonts}
\usepackage{graphicx, epsfig, subfigure}
\usepackage{enumerate} %% 改变列表标号样式宏包 其后可接选项[a,A,i,I,1]
\usepackage[square, numbers, sort&compress]{natbib} %% 支持引用的宏包
\ifx\pdfoutput\undefined
 \usepackage[dvipdfm,%
  pdfstartview=FitH, %% FitH表示打开pdf文件的时候自动缩放页面大小到视窗大小，Fits表示页宽适应windows大小
  bookmarks=true,%
  bookmarksnumbered=true, %% 把章节的标号放入书签中，缺省为flase
  bookmarksopen=true, %% 打开书签树
  plainpages=false,%
  pdfpagelabels,%
  colorlinks=true, %% 彩色文字链接打开, 去掉方框
  linkcolor=blue, %% 目录链接的颜色为蓝色
  citecolor=blue,%
  urlcolor=black,
  pdfborder=001]{hyperref}
  \AtBeginDvi{}  %  GBK  ->  Unicode
\else
 \usepackage[pdftex,%
  pdfstartview=FitH, %% FitH表示打开pdf文件的时候自动缩放页面大小到视窗大小，Fits表示页宽适应windows大小
  bookmarks=true,%
  bookmarksnumbered=true, %% 把章节的标号放入书签中，缺省为flase
  bookmarksopen=true, %% 打开书签树
  plainpages=false,%
  pdfpagelabels,%
  colorlinks=true, %% 彩色文字链接打开
  linkcolor=blue, %% 目录链接的颜色为蓝色
  citecolor=blue,%
  urlcolor=black,
  pdfborder=001]{hyperref}%%% 不能和cite宏包同时使用
\fi

\numberwithin{equation}{section}

\newcommand{\gx}{G^{\times}}

\begin{document}
\title{On total colorings of 1-planar graphs\footnotetext{Email addresses: xzhang@xidian.edu.cn~(X.\,Zhang), jfhou@fzu.edu.cn~(J.\,Hou), gzliu@sdu.edu.cn~(G.\,Liu)}\thanks{This research is mainly supported by the Fundamental Research Funds for the Central Universities (No.\,K5051370003), and is partially supported by the NSFC grants 11001055, 11101243, 11201440, 61070230 and the NSFFP grants 2010J05004, 2011J06001.}}

\author{Xin Zhang$^{1}$,\;\;Jianfeng Hou$^2$,\;\;Guizhen Liu$^3$\\[.5em]
{\small $^1$ Department of Mathematics, Xidian University}\\[-.1em]
{\small Xi'an, 710071, P. R. China}\\[-.1em]
{\small $^2$ Center for Discrete Mathematics, Fuzhou University}\\[-.1em]
{\small Fuzhou, 350002, P. R. China}\\[-.1em]
{\small $^3$ School of Mathematics, Shandong University}\\[-.1em]
{\small Jinan, 250100, P. R. China}\\[-.1em]
}
\date{}
\maketitle

\begin{abstract}\baselineskip 0.60cm
A graph is $1$-planar if it can be drawn on the plane so that each edge is crossed by at most one other edge. In this paper,
we confirm the total-coloring conjecture for 1-planar graphs with maximum degree at least 13.\\[.5em]
\end{abstract}
\baselineskip 0.60cm

\section{Introduction}

All graphs considered in this paper are finite, simple and undirected. We use $V(G)$, $E(G)$, $\delta(G)$ and $\Delta(G)$ to denote the vertex set, the edge set, the minimum degree and the maximum degree of a graph $G$, respectively. For a vertex $v\in V(G)$, $N_G(v)$ denotes the set of vertices that are adjacent to $v$ in $G$. By $d_G(v):=|N_G(v)|$, we denote the degree of $v$ in $G$. For a plane graph $G$, $F(G)$ denotes its face set and $d_G(f)$ denotes the degree of a face $f$ in $G$. Throughout this paper, a $k$-, $k^+$- and $k^-$-vertex (resp. face) is a vertex (resp. face) of degree $k$, at least $k$ and at most $k$. Any undefined notation follows that of Bondy and Murty \cite{Bondy}.

Given a graph $G$ and a positive integer $k$, a \emph{total $k$-coloring} of $G$ is a mapping from $V(G)\cup E(G)$ to $\{1,2,\cdots,k\}$ such that $f(x)\neq f(y)$ for every pair of adjacent or incident elements $x,y\in V(G)\cup E(G)$. The \emph{total chromatic number} $\chi''(G)$ of a graph $G$ is the least number of colors needed in any total coloring of $G$. It is clear that $\chi''(G)\geq \Delta(G)+1$. The next step is to look for any Brooks-typed or Vizing-typed upper bound on the total chromatic number in terms of maximum degree. However, to obtain such bounds turns out to be a difficult problem and has eluded mathematicians for nearly fifty years. The most well-known speculation is the \emph{total-coloring conjecture}, independently raised by Behzad \cite{Behzad.1965} and Vizing \cite{Vizing.1968}, which asserts that every graph of maximum degree $\Delta$ admits a total $(\Delta+2)$-coloring. The validity of this conjecture is known to be true for graphs in several wide families. Rosenfeld \cite{Rosenfeld.1971} and Vijayaditya \cite{Vijayaditya.1971} confirmed it for $\Delta\leq 3$, Kostochka solved it for $\Delta=4$ \cite{Kostochka.1977} and $\Delta=5$ \cite{Kostochka.1996}. For $\Delta\geq 6$ it remains open even for planar graphs, but more is known. Borodin \cite{Borodin.1989} confirmed the total-coloring conjecture for planar graphs with $\Delta\geq 9$. Yap \cite{Yap.1989} proved it for planar graphs with $\Delta\geq 8$. The $\Delta=7$ case was solved for planar graphs by Sanders and Zhao \cite{Sanders.1999}.

A graph is \emph{$1$-planar} if it can be drawn on the plane so that each edge is crossed by at most one other edge. The notion of $1$-planar graphs was introduced by Ringel \cite{Ringel.1965} while trying to simultaneously color the vertices and faces of a plane graph $G$ such that any pair of adjacent/incident elements receive different colors.
Various colorings including vertex coloring \cite{Ringel.1965,Borodin.1984,Borodin.1995}, list vertex coloring \cite{Albertson.2006,WL08}, acyclic vertex coloring \cite{Borodin.2001}, edge coloring \cite{Zhang.C5,Zhang.C4,Zhang.2010.SDU,Zhang.2011}, acyclic edge coloring \cite{Zhang.BKMS}, list edge and list total coloring\cite{list}, $(p,1)$-total labelling \cite{Zhang.CEJM} and the linear arboricity \cite{Zhang.LA} of 1-planar graphs have been extensively studied in the literature. In particular, Zhang, Wu and Liu \cite{list} proved that every 1-planar graph with maximum degree $\Delta\geq 16$ is $(\Delta+2)$-total choosable, which implies that the total-coloring conjecture holds for 1-planar graphs with maximum degree at least 16. In this paper, we improve the lower bound for the maximum degree in the above corollary to 13 by the following theorem.
\begin{thm}\label{mainthm}
Let $G$ be a $1$-planar graph with maximum degree $\Delta$ and let $r$ be an integer. If $\Delta\leq r$ and $r\geq 13$, then $\chi''(G)\leq r+2$.
\end{thm}
During the proof the Theorem \ref{mainthm}, we use the discharging method, and in particular, we involve an unusual approach to estimate the final charges of big vertices. This can be seen from Section 3.

\section{Structural Properties of a minimal 1-planar graph}

Let an \emph{$r$-minimal graph} be a connected graph $G$ on the fewest edges that has no total $(r+2)$-colorings. In the following lemmas, we always assume that $r\geq 13$.

\begin{lem}\label{degre-sum}
Let $G$ be a $r$-minimal graph and let $uv$ be an edge in $G$. If $d_G(u)\leq \lfloor\frac{r}{2}\rfloor$, then $d_G(u)+d_G(v)\geq r+3$.
\end{lem}

\begin{proof}
Suppose, to the contrary, that $d_G(u)+d_G(v)\leq r+2$. Since $G$ is $r$-minimal, the graph $G'=G-uv$ has a total $(r+2)$-coloring $\varphi$. First of all, erase the color of $u$ from $\varphi$. Since $d_{G'}(u)+d_{G'}(v)\leq \Delta(G)+2-2=\Delta(G)\leq r$, the uncolored edge $uv$ is incident with at most $r$ colored edges and one colored vertex, thus we can properly color $uv$ with a color involved in $\varphi$. At last, the vertex $u$ can be easily colored since it is incident with at most $2d_G(u)\leq r$ colors.
\end{proof}

\begin{lem}\label{3-vertex}
Let $G$ be a $r$-minimal graph and let $v$ be a vertex of $G$. If $d_G(v)=3$, then $v$ cannot be contained in a triangle.
\end{lem}

\begin{proof}
Let $N_G(v)=\{v_1,v_2,v_3\}$.
Suppose, to the contrary, that $v$ is contained in a triangle $vv_2v_3$. By the choice of $G$, the graph $G'=G-vv_3$ has a total $(r+2)$-coloring $\varphi$ with $\varphi(vv_i)=i$ for $i=1,2$. Now erase the color of $v$ from $\varphi$. For any color $i\geq 3$, $i$ must appear on $v_3$ or on some edge incident with $v_3$, since otherwise, we can color $vv_3$ with $i$, a contradiction. Thus, the colors 1 and 2 cannot appear on $v_3$ or the edges incident with $v_3$.
Now uncolor $vv_2$ and color $vv_3$ with 2. By the same argument, any color $i\ge 3$ must appear on $v_2$ or the edges incident with $v_2$ and the colors 1 and 2 cannot appear on there.
Now recolor $v_2v_3$ with 1, color $vv_3$ with $\varphi(v_2v_3)$ and color $vv_2$ with 2. At last, the vertex $v$ can be easily colored since it is adjacent or incident with at most 6 colors.
\end{proof}

\begin{lem}\label{4-vertex}
Let $G$ be a $r$-minimal graph and let $v$ be a $4$-vertex of $G$ with $N_G(v)=\{v_1,v_2,v_3,v_4\}$. For any $1\le i\le 4$, the edge $vv_i$ cannot be contained in two triangles.
\end{lem}

\begin{proof}
Suppose, to the contrary, that the edge $vv_4$ is contained in two triangles $vv_1v_4$ and $vv_3v_4$. By the choice of $G$, the graph $G'=G-vv_4$ has a total $(r+2)$-coloring $\varphi$ with $\varphi(vv_i)=i$ for $i=1,2,3$. Now erase the color of $v$ from $\varphi$. For any vertex $v$ in $G'$, let $S_{\varphi}(v)$ denote the set of colors not appearing on $v$ or the edges incident with $v$. First of all, we have $i\not\in S_{\varphi}(v_4)$ for any color $i\geq 4$, since otherwise, we can color $vv_4$ with $i$ and then the vertex $v$ can be easily colored (in the following we would not mention the coloring of $v$ for the last step). This implies that $S_{\varphi}(v_4)\subseteq \{1,2,3\}$. Note that $|S_{\varphi}(v_4)|\geq 2$.

\vspace{3mm}\noindent\textbf{Claim.} $S_{\varphi}(v_4)=\{1,3\}$

\begin{proof}
Otherwise, assume that $1\notin S_{\varphi}(v_4)$. This implies that $S_{\varphi}(v_4)=\{2,3\}$. Since $\varphi$ is a proper total coloring of $G'$, we may assume that $\varphi(v_1v_4)=4$.
If $i\in S_{\varphi}(v_1)$ for some $i\in \{2,3\}$, then recolor $v_1v_4$ with $i$ and color $vv_4$ with 4.
Otherwise, there is a color $i_0\ge 5$ such that $i_0\in S_{\varphi}(v_1)$. Note that 1 must appear on $v_2$ (resp.\,$v_3$) or edges incident with $v_2$ (resp.\,$v_3$), since otherwise, we can color recolor $vv_2$ (resp.\,$vv_3$) with 1, recolor $vv_1$ with $i_0$, and color $vv_4$ with 2 (resp.\,3). Moreover, for any $i\ge 4$, the color $i$ must appear on $v_2$ (resp.\,$v_3$), since otherwise, we can color $vv_2$ (resp.\,$vv_3$) with $i$ and color $vv_4$ with 2 (resp.\,3). This implies that $3\in S_{\varphi}(v_2)$ and $2\in S_{\varphi}(v_3)$.
Now we consider the color on $v_3v_4$. If $\varphi(v_3v_4)\neq 1$, then recolor $v_3v_4$ with 2 and color $vv_4$ with $\varphi(v_3v_4)$. Otherwise, $\varphi(v_3v_4)= 1$. In this case, recolor $vv_3$, $v_1v_4$ with 1, $v_3v_4$ with 3, $vv_1$ with $i_0$ and color $vv_4$ with 4.
\end{proof}

By the above claim, one can see that one of the edges $v_1v_4$ and $v_3v_4$ shall be colored with a color $i\geq 4$. Without loss of generality, assume that
$\varphi(v_1v_4)=4$. Note that $3\notin S_{\varphi}(v_1)$, since otherwise, we can recolor $v_1v_4$ with 3 and color $vv_4$ with 4. Moreover, $1\notin S_{\varphi}(v_3)$, since otherwise, we can exchange the colors on $v_1v_4$ and $v_1v$, then recolor $vv_3$ with 1 and color $vv_4$ with 3. For any $i\geq 4$, the color $i\notin S_{\varphi}(v_j)$ for any $j=1,3$, since otherwise, we can recolor $vv_j$ with $i$ and color $vv_4$ with $j$. Thus $S_{\varphi}(v_1)=S_{\varphi}(v_3)=\{2\}$. If there is a color $i\ge 4$ such that $i\in S_{\varphi}(v_2)$, then we can recolor $vv_2$ with $i$, $vv_1$ with 2, and color $vv_4$ with 1. Otherwise, we have $S_{\varphi}(v_2)\subseteq \{1,3\}$. Without loss of generality, let $1\in S_{\varphi}(v_2)$. Then we recolor $vv_2$  and $v_1v_4$ with 1, $vv_1$ with 2, and color $vv_4$ with 4.
\end{proof}

\begin{lem}\label{3-mas}
Let $G$ be a $r$-minimal graph and let $V_i$ be the set of $i$-vertices in $G$. We have $|V_{\Delta}|>2|V_3|$.
\end{lem}

\begin{proof}
If $|V_3|=0$, then it is trivial. If $|V_3|\not=0$, then by Lemma \ref{degre-sum}, $r=\Delta$.
Let $E$ be the set of edges in $G$ having one end-vertex in $V_3$ and let $H$ be the bipartite subgraph with vertex set $V_3\cup V_{\Delta}$ and edge set $E$. First of all, we prove that $H$ is a forest. Suppose, to the contrary, that $H$ contains a cycle $C$. Then this cycle is of even length in which alternate vertices have degree $3$ in $G$. Since $G$ is $\Delta$-minimal, the graph $G'=G-E(C)$ has a total $(\Delta+2)$-coloring $\varphi$. Now erase the colors of the 3-vertices on $C$ from $\varphi$. Let $e$ be an arbitrary edge of $C$. One can see that $e$ is now incident with at most $\Delta-1$ colored edges and one colored vertex, hence there are at least $(\Delta+2)-(\Delta-1+1)=2$ available colors for $e$. Therefore, the edges in $E(C)$ can be properly colored since every even cycle is 2-edge-choosable. At last, the 3-vertices on $C$ can be colored since each of them is now incident with at most six colored elements and no two of them are adjacent in $G$ by Lemma \ref{degre-sum}. This contradiction implies that $H$ is a forest and thus $|V(H)|=|V_3|+|V_{\Delta}|>|E(H)|$. Moreover, the neighbors of every vertex in $V_3$ belong to the vertex set $V_{\Delta}$ by Lemma \ref{degre-sum}. This implies that $|E(H)|=3|V_3|$. Hence we conclude that $|V_{\Delta}|>2|V_3|$.
\end{proof}

In the following, we restrict the minimal graph $G$ to be a 1-planar graph and assume that $G$ has already been embedded on a plane so that every edge is crossed by at most one other edge and the number of crossings is as small as possible.
The {\it associated plane graph} $\gx$ of $G$ is the plane graph that is obtained from $G$ by turning all crossings of $G$ into new $4$-vertices. A vertex in $\gx$ is {\it false} if it is not a vertex of $G$ and {\it true} otherwise. By a {\it false} face, we mean a face $f$ in $\gx$ that is incident with at least one false vertex; otherwise, we call $f$ {\it true}.

\begin{lem}{\rm\cite{Zhang.2011}}\label{3-incident-5+f}
Let $v$ be a $3$-vertex in $G$. If $v$ is incident with two false $3$-faces $vv_1v_2$ and $vv_1v_3$ in $\gx$, then $v_2$ and $v_3$ are both false and $v$ is incident with a $5^+$-face in $\gx$.
\end{lem}

%\begin{lemma}
%Every 3-vertex in $G$ is incident with at most two 3-faces in $\gx$.
%\end{lemma}
%
%\begin{proof}
%Let $v$ be a 3-vertex in $G$. Suppose, to the contrary, that $v$ is incident with three 3-faces. By Lemma \ref{3-no-c3}, those faces are false and so each of them is incident with one false vertex. However, this is impossible since no two false vertices are adjacent in $\gx$.
%\end{proof}

\begin{lem}\label{4-v-adj-4-face}
Every $4$-vertex in $G$ is incident with at most three $3$-faces in $\gx$.
\end{lem}

\begin{proof}
Let $v$ be a 4-vertex in $G$ and let $v_1,v_2,v_3,v_4$ be the neighbors in $\gx$ of $v$ that occurs clockwise around $v$. Suppose that $v$ is incident with four 3-faces in $\gx$. Then $v_1v_2,v_2v_3,v_3v_4,v_4v_1\in E(\gx)$. Since no two false vertices are adjacent in $\gx$, there are at most two false vertices among $v_1,v_2,v_3$ and $v_4$. If two of them, say $v_1$ and $v_3$, are false, then we would find two edges in $G$ that connect $v_2$ to $v_4$: one goes through the point $v_1$ and the other goes through the point $v_3$, contradicting the fact that $G$ is simple. Thus we shall assume that there are at least three true vertices, say $v_1,v_2$ and $v_3$, among the four neighbors of $v$. However, this is impossible by Lemma \ref{4-vertex} since $vv_1v_2$ and $vv_2v_3$ are two adjacent triangles in $G$ with $d_G(v)=4$.
\end{proof}

\begin{lem}\label{5-vertex}
Every $5$-vertex in $G$ is either incident with at least two $4^+$-faces in $\gx$, or adjacent to at least three true vertices in $\gx$, or
incident with one $4^+$-face and adjacent to two true vertices in $\gx$.
\end{lem}

\begin{proof}
Let $v$ be a 5-vertex in $G$ and let $v_1,v_2,v_3,v_4,v_5$ be the neighbors in $\gx$ of $v$ that occurs clockwise around $v$. Suppose that $v$ is incident with at most one $4^+$-face and adjacent to at most two true vertices in $\gx$. Without loss of generality, assume that $v_1v_2,v_2v_3,v_3v_4,v_4v_5\in E(\gx)$. Since no two false vertices are adjacent in $\gx$, there are at most three false vertices among $v_1,v_2,v_3,v_4$ and $v_5$. This implies that $v$ is adjacent to exactly
two true vertices in $\gx$. On the other hand, $v$ is incident with exactly one $4^+$-face because otherwise $v_1v_2v_3v_4v_5$ would be a 5-cycle in $\gx$, which implies that at least three of those five vertices are true, a contradiction to our assumption.
\end{proof}

\begin{lem}\label{5-face}
Every $5$-face in $\gx$ is incident with at most four $4^-$-vertices.
\end{lem}

\begin{proof}
Suppose, to the contrary, that the 5-face $f$ is incident only with $4^-$-vertices in $\gx$. Then $f$ is incident with at least three false vertices, because otherwise we would find an edge $uv$ on $f$ such that $u$ and $v$ are both true $4^-$-vertices, which is impossible by Lemma \ref{degre-sum}. On the other hand, $f$ can be incident with at most two false vertices since no two false vertices are adjacent in $\gx$. This contradiction completes the proof.
\end{proof}

\section{The proof of Theorem \ref{mainthm}}

We call a vertex $v$ in $\gx$ \emph{small} if $d_{\gx}(v)\leq 5$. Note that the degree of a false vertex in $\gx$ is four, so every false vertex is small.
We call $u$ the \emph{tri-neighbor} of $v$ if $uv$ is an edge of $G$ with $d_G(v)=4$ and $uv$ is incident with a 3-face $uvw$ in $\gx$ so that $w$ is true. Note that in this situation $u$ cannot be a tri-neighbor of $w$ by Lemma \ref{degre-sum}. Now we start to prove Theorem \ref{mainthm}.

Suppose that $G$ is a minimum counterexample to it. We then have that $G$ is 2-connected and moreover, $\delta(G)\geq 3$ by Lemma \ref{degre-sum}. In the following, we apply the discharging method to the associated plane graph $\gx$ of $G$ and complete the proof by contradiction. Note that $\gx$ is also 2-connected.

We now assign an initial charge $c$ to each element $x\in V(\gx)\cup F(\gx)$ as follows. If $x\in V(\gx)$, then let $c(x)=d_{\gx}(x)-6$. If $x\in F(\gx)$, then let $c(x)=2d_{\gx}(x)-6$.
Since $\gx$ is a planar graph,
$\sum_{x\in V(\gx)\cup F(\gx)}c(x)=-12$ by the well-known Euler's formula. We redistribute the initial charges on $V(\gx)\cup F(\gx)$ by the discharging rules below. Let $c'(x)$ be the final charge of an element $x\in V(\gx)\cup F(\gx)$ after discharging.
We still have $\sum_{x\in V(\gx)\cup F(\gx)}c'(x)=-12<0$, since our rules only move charge around and do not affect the sum.\\[.5em]
\noindent R1. Every $4^+$-face redistributes its initial charge uniformly among the small vertices that are incident with it in $\gx$.\\
\noindent R2. Every $\Delta$-vertex gives $\frac{1}{2}$ to a common pot from which each 3-vertex receives 1, if $|V_3|>0$.\\
\noindent R3. Let $u,v$ be true vertices of $\gx$ and let $uv\in E(\gx)$. If $v$ is small, then $u$ sends $\frac{1}{3}$ to $v$; moreover, if $u$ is a tri-neighbor of $v$, then $u$ sends an addition of $\frac{1}{12}$ to $v$.\\[.5em]
Note that in R2, the common pot can also be seen as a pseudo-point that has initial charge zero.
In the next six rules, we assume that $uv$ crosses $xy$ at a false vertex $w$ in $\gx$ there.\\[.5em]
\noindent R4. If $d_{\gx}(u)\geq 9$, $ux, uy\not\in E(\gx)$ and $v$ is a small vertex, then $u$ sends $\frac{1}{3}$ to $v$ through $w$.\\
\noindent R5. If $d_{\gx}(u)\geq 9$, $ux\not\in E(\gx)$ and $uy\in E(\gx)$, then $u$ sends $\frac{1}{4}$ to $w$. Furthermore, if $d_{\gx}(v)\leq 4$, then $u$ sends $\frac{1}{3}$ to $v$ through $w$.\\
\noindent R6. If $d_{\gx}(u)\geq 9$, $ux, uy, vx\in E(\gx)$ and $y$ is a small vertex, then $u$ sends $\frac{3}{4}$ to $w$. Furthermore, if $d_{\gx}(v)\leq 4$, then $u$ sends $\frac{1}{24}$ to $v$ through $w$.\\
\noindent R7. If $d_{\gx}(u)\geq 9$, $ux, uy\in E(\gx)$ and either $vx\not\in E(\gx)$ or $y$ is not a small vertex, then $u$ sends $\frac{2}{3}$ to $w$. Furthermore, if $d_{\gx}(v)\leq 4$, then $u$ sends $\frac{1}{8}$ to $v$ through $w$.\\
\noindent R8. If $d_{\gx}(u)=8$ and $ux, uy\in E(\gx)$, then $u$ sends $\frac{1}{2}$ to $w$.\\
\noindent R9. If $d_{\gx}(u)=8$, $ux\in E(\gx)$ and $uy\not\in E(\gx)$, then $u$ sends $\frac{1}{12}$ to $w$.\\[.5em]
In the following, we check that the final charge $c'$ on each vertex and face is nonnegative. And we also show that the final charge of the common pot is nonnegative. This implies that $\sum_{x\in V(\gx)\cup F(\gx)}c'(x)\geq 0$, a contradiction.

First of all, since $|V_{\Delta}|>2|V_3|$ by Lemma \ref{3-mas}, the final charge of the common pot is at least $\frac{1}{2}|V_{\Delta}|-|V_3|>0$ by R2. One can also check that the final charge of every face in $F(\gx)$ is exactly 0 by R1. Thus in the following we consider the vertices in $\gx$.

Let $v$ be a $d$-vertex in $\gx$ and let $v_1,v_2,\cdots,v_d$ be its neighbors in $\gx$ that occur around $v$ in a clockwise order. By $f_i$ denote the face incident with $vv_i$ and $vv_{i+1}$ in $\gx$, where the addition on subscripts are taken modulo $d$.

\textbf{Case 1.} $d=3$.

\textbf{Case 1.1.} If $v$ is adjacent to at most one false vertex in $\gx$, then without loss of generality assume that $v_2$ and $v_3$ are true. By Lemmas \ref{degre-sum} and \ref{3-vertex}, neither $v_2$ nor $v_3$ is small and $f_2$ is a $4^+$-face. Thus by R1 and R3, $v$ receives at least $2\times\frac{1}{3}+\frac{2}{4-2}=\frac{5}{3}$ from $v_2,v_3$ and $f_2$. By Lemmas \ref{3-vertex} and \ref{3-incident-5+f}, at least one of $f_1$ and $f_3$, say $f_1$, shall be a $4^+$-face. Then by R1, $f_1$ sends at least $\frac{2}{4-1}=\frac{2}{3}$ to $v$. Furthermore, $v$ would receive 1 from the common pot by R2. Therefore, $c'(v)\geq -3+\frac{5}{3}+\frac{2}{3}+1>0$.

\textbf{Case 1.2.} If $v$ is adjacent to two false vertices in $\gx$, say $v_1$ and $v_2$, then $f_1$ is a $4^+$-face since $v_1v_2\not\in E(\gx)$. By R1 and R3, $v$ receives a total of $1+\frac{1}{3}=\frac{4}{3}$ from the common pot and $v_3$. Now we consider three subcases.

First, assume that $f_2$ and $f_3$ are both $4^+$-faces. Then by R1, $f_1$, $f_2$ and $f_3$ sends at least $\frac{2}{4}=\frac{1}{2}$, $\frac{2}{4-1}=\frac{2}{3}$ and $\frac{2}{4-1}=\frac{2}{3}$ to $v$, respectively. Therefore, $c'(v)\geq -3+\frac{4}{3}+\frac{1}{2}+\frac{2}{3}+\frac{2}{3}>0$.

Second, assume that $f_2$ is a $4^+$-face and $f_3$ is a 3-face. Let $v'_1$ be a vertex such that $vv'_1$ is an edge in $G$ that goes through the false vertex $v_1$ in $\gx$. Then by Lemmas \ref{degre-sum} and \ref{3-vertex}, $v'_1$ is a $\Delta$-vertex and $v'_1v_3\not\in E(\gx)$, because otherwise $vv'_1v_3$ would be a triangle in $G$. Thus by R4 and R5, $v$ receives $\frac{1}{3}$ from $v'_1$.
If $f_2$ is a $5^+$-face, then by R1, $f_2$ sends at least $\frac{4}{5-1}=1$ to $v$ (note that $v_3$ is not a small vertex). Since $f_1$ is a $4^+$-face, $f_1$ would send at least $\frac{2}{4}=\frac{1}{2}$ to $v$ by R1. Thus, $c'(v)\geq -3+\frac{4}{3}+\frac{1}{3}+1+\frac{1}{2}>0$. So we suppose that $f_2$ is a 4-face, from which $v$ receives at least $\frac{2}{4-1}=\frac{2}{3}$ by R1. If $f_1$ is a $5^+$-face, then by R1, $f_1$ sends at least $\frac{4}{5}$ to $v$. Thus $c'(v)\geq -3+\frac{4}{3}+\frac{1}{3}+\frac{2}{3}+\frac{4}{5}>0$. So suppose that $f_1$ is a $4$-face. Let $v'_2$ and $v'_3$ be the fourth (undefined) vertex on $f_2$ and $f_1$, respectively. Since $v_2$ is false and $v_2v'_2,v_2v'_3\in E(\gx)$, $v'_2v'_3$ is an edge in $G$. By Lemma \ref{degre-sum}, one of $v'_2$ and $v'_3$ is not small. If $v'_2$ is not small, then by R1, $f_1$ and $f_2$ sends at least $\frac{2}{4}=\frac{1}{2}$ and $\frac{2}{4-2}=1$ to $v$, respectively. It follows that $c'(v)\geq -3+\frac{4}{3}+\frac{1}{3}+\frac{1}{2}+1>0$. If $v'_3$ is not small, then by R1, $f_1$ and $f_2$ sends at least $\frac{2}{4-1}=\frac{2}{3}$ and $\frac{2}{4-1}=\frac{2}{3}$ to $v$, respectively. It follows that $c'(v)\geq -3+\frac{4}{3}+\frac{1}{3}+\frac{2}{3}+\frac{2}{3}=0$.

Third, assume that $f_2$ and $f_3$ are both $3$-faces. Then by Lemma \ref{3-incident-5+f}, $f_1$ is a $5^+$-face. Let $v'_i~(i=1,2)$ be a vertex such that $vv'_i$ is an edge in $G$ that goes through the false vertex $v_i$ in $\gx$. By a similar argument as the beginning of the second subcase above, one can prove that $v$ receives $\frac{1}{3}$ from each of $v'_1$ and $v'_2$. If $f_1$ is a $6^+$-face, then by R1, $f_1$ sends at least $\frac{6}{6}=1$ to $v$.
If $f_1$ is a 5-face, then assume that $v_3x_1$ crosses $vv'_1$ and $v_3x_2$ crosses $vv'_2$ in $G$. It follows that $x_1x_2\in E(G)$. By Lemma \ref{degre-sum}, at least one of $x_1$ and $x_2$ is not small. Thus by R1, $f_1$ sends at least $\frac{4}{5-1}=1$ to $v$. In each case we have $c'(v)\geq -3+\frac{4}{3}+2\times\frac{1}{3}+1=0$.

\textbf{Case 1.3.} If $v$ is adjacent to three false vertices in $\gx$, then $f_1,f_2$ and $f_3$ are $4^+$-faces. By R2, $v$ receives 1 from the common pot.
If two of $f_1,f_2$ and $f_3$ are of degree at least 5, then by R1 it is easy to calculate that $v$ receives at least $\frac{4}{5}+\frac{4}{5}+\frac{2}{4}>2$ from its incident faces and therefore $c'(v)\geq -3+1+2=0$. If exactly one of $f_1,f_2$ and $f_3$, say $f_3$, is a $5^+$-face, then let $x_1$ and $x_2$ be the fourth (undefined) vertices of the 4-faces $f_1$ and $f_2$, respectively. One can easily see that $x_1x_2\in E(G)$ and thus by Lemma \ref{degre-sum}, at least one of $x_1$ and $x_2$ is not small. Therefore, $v$ receives at least $\frac{4}{5}+\frac{2}{4}+\frac{2}{4-1}=\frac{59}{30}$ from its incident faces by R1. Assume that $vv'_2$ crosses $x_1x_2$ in $G$, then by Lemma \ref{degre-sum}, $v'_2$ is a $\Delta$-vertex. Thus, $v'_2$ sends at least $\frac{1}{8}$ to $v$ by R4--R7. This implies that $c'(v)\geq -3+1+\frac{59}{30}+\frac{1}{8}>0$. If $f_1,f_2$ and $f_3$ are all $4$-faces, then let $x_i~(i=1,2,3)$ be the fourth (undefined) vertices of the 4-faces $f_i$. It is easy to check that $x_1x_2,x_2x_3,x_3x_1\in E(G)$ by the drawing of $G$. Thus, at most one of $x_1,x_2$ and $x_3$ is small by Lemma \ref{degre-sum}. This implies that $v$ receives at least $\frac{2}{4}+\frac{2}{4-1}+\frac{2}{4-1}=\frac{11}{6}$ form its incident faces by R1. Assume that $vv'_i~(i=1,2,3)$ crosses $v_{i-1}v_i$ in $G$, where the subscripts are taken modulo 3, then by Lemma \ref{degre-sum}, $v'_i$ is a $\Delta$-vertex, from which $v$ receives at least $\frac{1}{8}$ by R4--R7. Therefore, $c'(v)\geq -3+1+\frac{11}{6}+3\times\frac{1}{8}>0$.

\textbf{Case 2.} $d=4$ and $v$ is a true vertex.

By Lemma \ref{4-v-adj-4-face}, $v$ is incident with at least one $4^+$-face in $\gx$. Thus we consider four subcases.

\textbf{Case 2.1.} If $v$ is incident with four $4^+$-faces in $\gx$, then $v$ receives at least $\frac{2}{4}=\frac{1}{2}$ from each of its incident faces by R1. This implies that $c'(v)\geq -2+4\times\frac{1}{2}=0$.

\textbf{Case 2.2.} If $v$ is incident with exactly three $4^+$-faces in $\gx$, say $f_2,f_3$ and $f_4$, then $v_1v_2\in E(\gx)$. Since no two false vertices are adjacent in $\gx$, at least one of $v_1$ and $v_2$, say $v_1$, is true, and moreover, is a $12^+$-vertex by Lemma \ref{degre-sum}. So by R3 and R1, $v$ receives $\frac{1}{3}$ from $v_1$, at least $\frac{2}{4-1}=\frac{2}{3}$ from $f_4$ and at least $\frac{2}{4}=\frac{1}{2}$ from each of $f_2$ and $f_3$. Therefore, $c'(v)\geq -2+\frac{1}{3}+\frac{2}{3}+2\times\frac{1}{2}=0$.

\textbf{Case 2.3.} If $v$ is incident with exactly two $4^+$-faces in $\gx$, then we consider two subcases.

Assume first that $f_1$ and $f_3$ are both $4^+$-faces. Then by a same argument as in Case 2.2, at least one of $v_2$ and $v_3$ and at least one of $v_1$ and $v_4$ are $12^+$-vertices. If $v_1$ and $v_2$ are both $12^+$-vertices, then by R3 and R1, $v$ receives $\frac{1}{3}$ from each of $v_1$ and $v_2$, at least $\frac{2}{4-2}=1$ from $f_1$ and at least $\frac{2}{4}=\frac{1}{2}$ from $f_3$. Thus, $c'(v)\geq -2+2\times\frac{1}{3}+1+\frac{1}{2}>0$. If $v_1$ and $v_3$ are both $12^+$-vertices, then by R3 and R1, $v$ receives $\frac{1}{3}$ from each of $v_1$ and $v_3$ and at least $\frac{2}{4-1}=\frac{2}{3}$ from each of $f_1$ and $f_3$. This implies that $c'(v)\geq -2+2\times\frac{1}{3}+2\times\frac{2}{3}=0$.

Second, assume that $f_1$ and $f_2$ are $4^+$-faces. If $v_1$ and $v_3$ are both true, then by Lemma \ref{degre-sum} they are $12^+$-vertices. So by R3 and R1, $v$ receives $\frac{1}{3}$ from each of $v_1$ and $v_3$ and at least $\frac{2}{4-1}=\frac{2}{3}$ from each of $f_1$ and $f_2$. This implies that $c'(v)\geq -2+2\times\frac{1}{3}+2\times\frac{2}{3}=0$. So we assume that at least one of $v_1$ and $v_3$ is false, which implies that $v_4$ is true since no two false vertices are adjacent in $\gx$.

If $v_1$ is false and $v_3$ is true, then let $v'_1$ be the vertex of $G$ so that $vv'_1$ is a crossed edge in $G$ with a crossing $v_1$. By Lemma \ref{4-vertex}, $v'_1v_4\not\in E(G)$, because otherwise $vv_4v'_1$ and $vv_3v_4$ would be two adjacent triangles in $G$ with a common 4-vertex. Note that $v'_1$ and $v_3$ are $12^+$-vertices by Lemma \ref{degre-sum}. So $v$ receives $\frac{1}{3}$ from $v'_1$ by R4 and R5, $\frac{1}{3}$ from each of $v_3$ and $v_4$ by R3 and at least $\frac{2}{4}=\frac{1}{2}$ from each of $f_1$ and $f_2$ by R1. This implies that $c'(v)\geq -2+\frac{1}{3}+2\times\frac{1}{3}+2\times\frac{1}{2}=0$.

If $v_1$ and $v_3$ are both false, then let $v'_i$ and $x_i~(i=1,3)$ be the vertices of $G$ so that $vv'_i$ crosses $v_4x_i$ in $G$ at the crossing $v_i$.
Note that $v'_1$ and $v'_3$ are both $12^+$-vertices by Lemma \ref{degre-sum}. By Lemma \ref{4-vertex}, $v'_1v_4$ and $v'_3v_4$ cannot simultaneously be the edges of $G$, because otherwise $vv_4v'_1$ and $vv_4v'_3$ would be two adjacent triangles in $G$ with a common 4-vertex. Without loss of generality, assume that $v'_1v_4\not\in E(G)$. By R3, R4 and R5, each of $v'_1$ and $v_4$ sends $\frac{1}{3}$ to $v$ (recall that $v_4$ is true).
If $v_2$ is true, then $v$ receives $\frac{1}{3}$ from $v_2$ by R3. Moreover, each of $f_1$ and $f_2$ sends at least $\frac{2}{4}=\frac{1}{2}$ to $v$ by R1. Thus, $c'(v)\geq -2+\frac{1}{3}+2\times\frac{1}{3}+2\times\frac{1}{2}=0$.
If $v_2$ is false, then let $v'_2$ be the vertex of $G$ so that $vv'_2$ is a crossed edge in $G$ with a crossing $v_2$. By Lemma \ref{degre-sum}, $v'_2$ is a $12^+$-vertex. If at least one of $f_1$ and $f_2$, say $f_1$, is a $5^+$-face, then $f_1$ sends at least $\min\{\frac{6}{6}, \frac{4}{4}\}=1$ to $v$ by R1 and Lemma \ref{5-face} and $f_2$ sends at least $\frac{2}{4}=\frac{1}{2}$ to $v$ by R1. Thus, $c'(v)\geq -2+2\times\frac{1}{3}+1+\frac{1}{2}>0$. So we assume that $f_1$ and $f_2$ are both 4-faces. This implies that $x_1x_3$ is a crossed edge in $G$ with the crossing $v_2$. By Lemma \ref{degre-sum}, at most one of $x_1$ and $x_3$ is small. So $f_1$ and $f_2$ totally sends at least $\frac{2}{4-1}+\frac{2}{4}=\frac{7}{6}$ to $v$ by R1. Recall that $v'_2$ and $v'_3$ are $12^+$-vertices. By R4--R7, $v'_2$ sends at least $\frac{1}{8}$ and $v'_3$ sends at least $\frac{1}{24}$ to $v$. Therefore, $c'(v)\geq -2+2\times\frac{1}{3}+\frac{7}{6}+\frac{1}{8}+\frac{1}{24}=0$.

\textbf{Case 2.4.} If $v$ is incident with exactly one $4^+$-faces in $\gx$, say $f_1$, then $v_2v_3,v_3v_4,v_4v_1\in E(\gx)$. Now we claim that at least one of $v_1$ and $v_2$ is false. Suppose, to the contrary, that $v_1$ and $v_2$ are true vertices. If $v_3$ is true, then either $vv_3v_4$ (when $v_4$ is true) or $vv_1v_3$ (when $v_4$ is false) is a triangle in $G$ that is adjacent to another triangle $vv_2v_3$, which is impossible by Lemma \ref{4-vertex}. Thus we shall assume that $v_3$ is false. By symmetry, $v_4$ is also false, but it contradicts the fact that $v_3v_4\in E(\gx)$. Without loss of generality, assume that $v_1$ is false. It follows that $v_4$ is a true vertex. By Lemma \ref{4-vertex}, exactly one of $v_2$ and $v_3$ shall be false, because otherwise $vv_2v_3$ and $vv_3v_4$ would be two adjacent triangles in $G$ with a common 4-vertex. Thus we consider two subcases.

Assume first that $v_2$ is false and $v_3$ is true. One can check that $v_3$ and $v_4$ are both tri-neighbors of $v$, which follows that each of $v_3$ and $v_4$ sends $\frac{1}{3}+\frac{1}{12}=\frac{5}{12}$ to $v$ by R3.
Let $v'_i~(i=1,2)$ be the vertex of $G$ so that $vv'_i$ is a crossed edge in $G$ with the crossing $v_i$. It is easy to see that $v'_1$ and $v'_2$ are $12^+$-vertices by Lemma \ref{degre-sum}. One can also prove that $v'_1v_4,v'_2v_3\not\in E(G)$ by a similar argument as in Case 2.3. Thus by R4 and R5, each of $v'_1$ and $v'_2$ sends $\frac{1}{3}$ to $v$. Since $f_1$ is a $4^+$-face, $f_1$ sends at least $\frac{2}{4}=\frac{1}{2}$ to $v$ by R1. Therefore, $c'(v)\geq -2+2\times\frac{5}{12}+2\times\frac{1}{3}+\frac{1}{2}=0$.

Now assume that $v_2$ is true and $v_3$ is false. It is easy to see that $vv_2v_4$ is a triangle in $G$ by the drawing of $G$. Let $v'_i~(i=1,3)$ be the vertex of $G$ so that $vv'_i$ is a crossed edge in $G$ with the crossing $v_i$. One can see that $v'_1$ and $v'_3$ are $12^+$-vertices by Lemma \ref{degre-sum} and can prove that $v'_1v_4,v'_3v_4\not\in E(G)$ by a similar argument as in Case 2.3. So each of $v'_1$ and $v'_3$ sends $\frac{1}{3}$ to $v$ by R3 and R4. Meanwhile, each of $v_2$ and $v_4$ sends $\frac{1}{3}$ to $v$ by R3 and $f_1$ sends at least $\frac{2}{4-1}=\frac{2}{3}$ to $v$ by R1 (note that $v_2$ is not small). Therefore, $c'(v)\geq -2+2\times\frac{1}{3}+2\times\frac{1}{3}+\frac{2}{3}=0$.

\textbf{Case 3.} $d=4$ and $v$ is a false vertex.

\textbf{Case 3.1.} If $v$ is incident with no 3-faces in $\gx$, then by R1, each of $f_1,f_2,f_3$ and $f_4$ sends at least $\frac{2}{4}=\frac{1}{2}$ to $v$. So $c'(v)\geq -2+4\times\frac{1}{2}=0$.

\textbf{Case 3.2.} If $v$ is incident with exactly one 3-face, say $f_1$, then $v_1v_2\in E(G)$. This implies that at most one of $v_1$ and $v_2$ can be a $7^-$-vertex by Lemma \ref{degre-sum}. Assume first that $\min\{d_{\gx}(v_1),d_{\gx}(v_2)\}\geq 8$. Then by R1,each of $f_2$ and $f_4$ sends at least $\frac{2}{4-1}=\frac{2}{3}$ to $v$ and $f_3$ sends at least $\frac{2}{4}=\frac{1}{2}$ to $v$. Moreover, each of $v_1$ and $v_2$ sends at least $\frac{1}{12}$ to $v$ by R5 and R9. Thus $c'(v)\geq -2+2\times\frac{2}{3}+\frac{1}{2}+2\times\frac{1}{12}=0$. Now assume that $d_{\gx}(v_1)\leq 7$. It follows that $\min\{d_{\gx}(v_2),d_{\gx}(v_3)\}\geq 9$ by Lemma \ref{degre-sum}. Thus $f_2,f_3$ and $f_4$ sends at least $\frac{2}{4-2}=1$, $\frac{2}{4-1}=\frac{2}{3}$ and $\frac{2}{4}=\frac{1}{2}$ to $v$ by R1, respectively. Therefore, $c'(v)\geq -2+1+\frac{2}{3}+\frac{1}{2}>0$.

\textbf{Case 3.3.} If $v$ is incident with exactly two 3-faces, then we consider two subcases.

Assume first that $f_1$ and $f_2$ are both 3-faces. Then $v_1v_2,v_2v_3\in E(G)$. If $d_{\gx}(v_2)\leq 8$, then by Lemma \ref{degre-sum}, $\min\{d_{\gx}(v_1),d_{\gx}(v_3),d_{\gx}(v_4)\}\geq 7$. This implies that each of $f_3$ and $f_4$ sends at least $\frac{2}{4-2}=1$ to $v$ and thus $c'(v)\geq -2+2\times 1=0$. So we assume that $d_{\gx}(v_2)\geq 9$. It follows that $v_2$ sends $\frac{2}{3}$ to $v$ by R7. If one of $v_1$ and $v_3$, say $v_1$, is small, then by R1, $f_3$ and $f_4$ sends at least $\frac{2}{4-1}=\frac{2}{3}$ and $\frac{2}{4}=\frac{1}{2}$ to $v$, respectively, since in this case we also have $d_{\gx}(v_3)\geq 11$ by Lemma \ref{degre-sum}. Moreover, $v_3$ sends $\frac{1}{4}$ to $v$ by R5. Therefore, $c'(v)\geq -2+\frac{2}{3}+\frac{2}{3}+\frac{1}{2}+\frac{1}{4}>0$. On the other hand, if neither $v_1$ nor $v_3$ is small, then by R1, each of $f_3$ and $f_4$ sends at least $\frac{2}{4-1}=\frac{2}{3}$ to $v$. Thus $c'(v)\geq -2+\frac{2}{3}+2\times\frac{2}{3}=0$.

Now assume that $f_1$ and $f_3$ are both 3-faces. If none of $v_1$, $v_2$, $v_3$ and $v_4$ is small, then by R1, each of $f_2$ and $f_4$ sends at least $\frac{2}{4-2}=1$ to $v$, which implies that $c'(v)\geq -2+2\times 1=0$. If at least one of $v_1$, $v_2$, $v_3$ and $v_4$, say $v_1$, is small, then by Lemma \ref{degre-sum}, $\min\{d_{\gx}(v_2),d_{\gx}(v_3)\}\geq 11$. So $f_2$ and $f_4$ sends at least $\frac{2}{4-2}=1$ and $\frac{2}{4}=\frac{1}{2}$ to $v$ by R1, respectively. Moreover, each of $v_2$ and $v_3$ sends $\frac{1}{4}$ to $v$ by R5. Therefore, $c'(v)\geq -2+1+\frac{1}{2}+2\times\frac{1}{4}=0$.

\textbf{Case 3.4.} If $v$ is incident with exactly three 3-faces, say $f_1,f_2$ and $f_3$, then $v_1v_2,v_2v_3,v_3v_4\in E(G)$. If $d_{\gx}(v_2)\leq 7$, then by Lemma \ref{degre-sum}, $\min\{d_{\gx}(v_1),d_{\gx}(v_3),d_{\gx}(v_4)\}\geq 9$. So $f_4$ sends at least $\frac{2}{4-2}=1$ to $v$ by R1, each of $v_1$ and $v_4$ sends $\frac{1}{4}$ to $v$ by R5 and $v_3$ sends at least $\frac{2}{3}$ to $v$ by R6 and R7. Thus $c'(v)\geq -2+1+2\times\frac{1}{4}+\frac{2}{3}>0$. So we shall assume that $d_{\gx}(v_2)\geq 8$. Similarly, we shall assume that $d_{\gx}(v_3)\geq 8$. If both $v_1$ and $v_4$ are small, then by Lemma \ref{degre-sum}, $\min\{d_{\gx}(v_2),v_3\}\geq 11$. It follows that each of $v_2$ and $v_3$ sends $\frac{3}{4}$ to $v$ by R6. Moveover, $f_4$ sends at least $\frac{2}{4}=\frac{1}{2}$ to $v$. Thus $c'(v)\geq -2+2\times\frac{3}{4}+\frac{1}{2}=0$. So we assume that at least one of $v_1$ and $v_4$ is not small. It follows that $f_4$ sends at least $\frac{2}{4-1}=\frac{2}{3}$ to $v$ by R1. If $d_{\gx}(v_1)\leq 7$ or $d_{\gx}(v_4)\leq 7$, then by Lemma \ref{degre-sum}, $\min\{d_{\gx}(v_2),d_{\gx}(v_3)\}\geq 9$. So by R6 and R7, each of $v_2$ and $v_3$ sends at least $\frac{2}{3}$ to $v$. Thus $c'(v)\geq -2+\frac{2}{3}+2\times\frac{2}{3}=0$. So we shall assume that $\min\{d_{\gx}(v_1),d_{\gx}(v_4)\}\geq 8$. It follows that $f_4$ sends at least $\frac{2}{4-2}=1$ to $v$ by R1. Moreover, each of $v_2$ and $v_3$ sends at least $\frac{1}{2}$ to $v$ by R6, R7 and R8. Therefore, $c'(v)\geq -2+1+2\times\frac{1}{2}=0$.

\textbf{Case 3.5.} If $v$ is incident with four 3-faces, then $v_1v_2,v_2v_3,v_3v_4,v_4v_1\in E(G)$ and thus at most one of $v_1,v_2,v_3$ and $v_4$ is a $7^-$-vertex by Lemma \ref{degre-sum}. Assume first that $d_{\gx}(v_1)\leq 7$. Then all of $v_2,v_3$ and $v_4$ are $9^+$-vertices by Lemma \ref{degre-sum}.
So by R6 and R7, each of $v_2,v_3$ and $v_4$ sends at least $\frac{2}{3}$ to $v$, which implies that $c'(v)\geq -2+3\times\frac{2}{3}=0$. Now assume that all of $v_1,v_2,v_3$ and $v_4$ are $8^+$-vertices. Then by R6, R7 and R8, each of those four vertices sends at least $\frac{1}{2}$ to $v$. This implies that $c'(v)\geq -2+4\times\frac{1}{2}=0$.

\textbf{Case 4.} $d=5$.

By R1 and R3, $v$ receives at least $\frac{2}{4}=\frac{1}{2}$ from each of its incident $4^+$-faces and $\frac{1}{3}$ from each of its adjacent true vertices in $\gx$. We consider three subcases according to Lemma \ref{5-vertex}.
If $v$ is incident with at least two $4^+$-faces, then $c'(v)\geq -1+2\times\frac{1}{2}=0$.
If $v$ is adjacent to at least three trues vertices in $\gx$, then $c'(v)\geq -1+3\times\frac{1}{3}=0$.
If $v$ is incident with one $4^+$-face and adjacent to two true vertices in $\gx$, then $c'(v)\geq -1+\frac{1}{2}+2\times\frac{1}{3}>0$.

\textbf{Case 5.} $d\geq 6$.

If $d\leq 7$, then it is trivial that $c'(v)=c(v)\geq 0$, so we assume that $d\geq 8$.

Let $S_f(v)$ denote the subgraph induced by the faces that are incident with $v$ in $\gx$. Then $S_f(v)$ can be decomposed into many parts, each of which is one of the five clusters in Figure \ref{fig}, and any two parts of which are adjacent only if they have a coJPGmmon edge $vw$ such that $w$ is a true vertex. The hollow vertices in Figure \ref{fig} are false vertices and the solid ones are true vertices; all the marked faces are $4^+$-faces and there is at least one $4^+$-face contained in the clusters of type 2, 4 and 5.

\begin{figure}
\begin{center}
    \includegraphics[height=6.5cm]{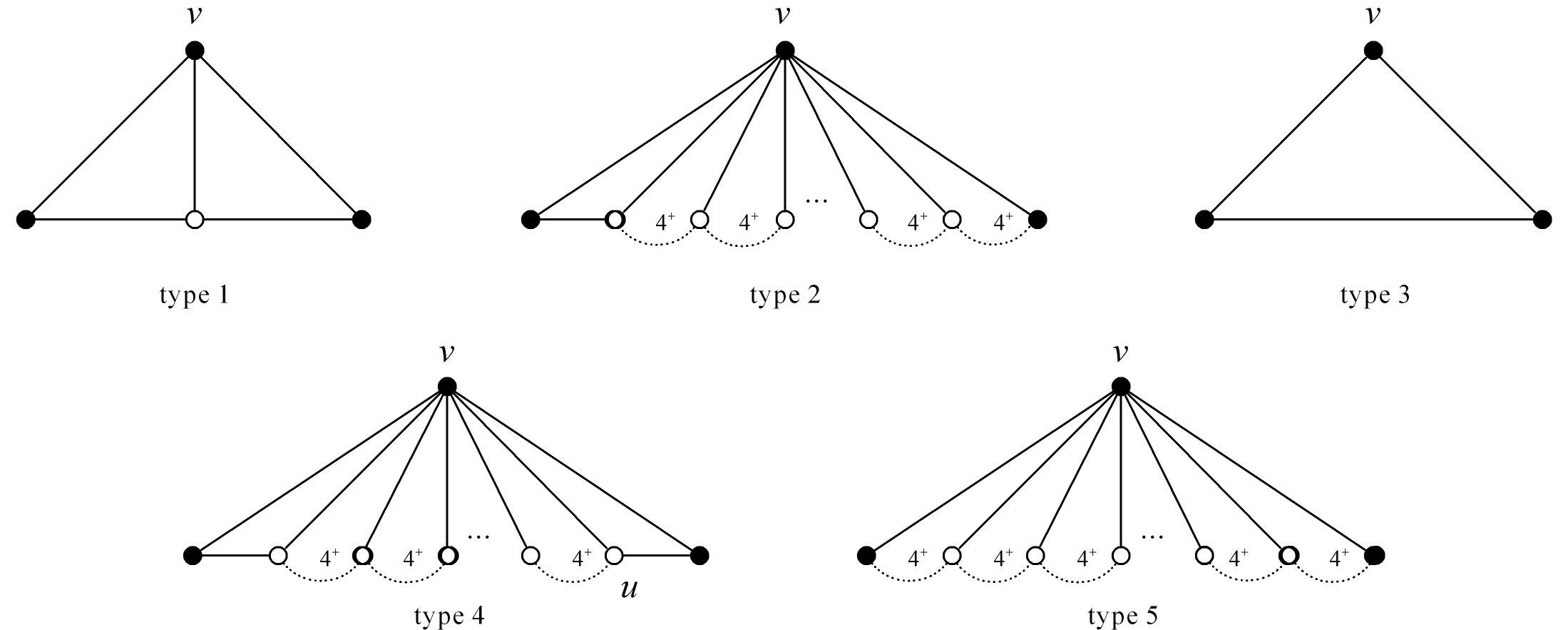}
\end{center}
    \caption{Five types of cluster}
    \label{fig}
\end{figure}

Let $a_i$ denote the largest possible value of the charges sent by $v$ to or through its adjacent false vertices in a cluster of type $i$.

If $d=8$, then by R8 and R9 we have $a_1=\frac{1}{2}$, $a_2=\frac{1}{12}$, $a_3=0$, $a_4=2\times\frac{1}{12}=\frac{1}{6}$ and $a_5=0$.

If $9\leq d\leq 11$, then by Lemma \ref{degre-sum}, $v$ is adjacent to no $4^-$-vertices in $G$. Thus by R4, R5, R6 and R7 we have $a_1=\frac{3}{4}$, $a_2=\frac{1}{4}$, $a_3=0$, $a_4=2\times\frac{1}{4}=\frac{1}{2}$ and $a_5=0$.

If $d\geq 12$, then $v$ may be adjacent to $4^-$-vertices in $G$, to which $v$ can send charges through the false vertices that are adjacent to $v$ in $\gx$. First of all,
$a_1=\max\{\frac{3}{4}+\frac{1}{24},\frac{2}{3}+\frac{1}{8}\}=\frac{19}{24}$ by R6 and R7 and $a_3=0$.
Let $H_i~(i=2,4,5)$ be a cluster of type $i$. Suppose that there are $s_i$ false vertices that are adjacent to $v$ in $H_i$. By R4 and R5, we have $a_2=\frac{1}{4}+\frac{1}{3}s_2$, $a_4=2\times\frac{1}{4}+\frac{1}{3}s_4=\frac{1}{2}+\frac{1}{3}s_4$ and $a_5=\frac{1}{3}s_5$.

Denote by $n_i$ the number of clusters of type $i$ contained in $S_f(v)$. Let $m$ be the total number of false vertices that are adjacent to $v$ in the clusters of type 2, 4 and 5.
One can easy to see that the following facts hold.

(1) $v$ is adjacent to $n_1+n_2+n_3+n_4+n_5$ true vertices in $\gx$.

(2) $v$ is adjacent to $n_1+m$ false vertices in $\gx$.

(3) $2n_1+2n_2+n_3+3n_4+n_5\leq d$.

By (1) and (2), it is easy to see that $m=d-2n_1-n_2-n_3-n_4-n_5$.

First of all, we calculate the largest possible value of the charges sent by $v$ to or through its adjacent false vertices in $\gx$, that is, the value of $n_1a_1+n_2a_2+n_3a_3+n_4a_4+n_5a_5$. Recall the values of $a_i$ we have obtained in each of the above cases. One can deduce that
\begin{align*}
n_1a_1+n_2a_2+n_3a_3+n_4a_4+n_5a_5&=\frac{1}{2}n_1+\frac{1}{12}n_2+\frac{1}{6}n_4
\end{align*}
if $d=8$,
\begin{align*}
n_1a_1+n_2a_2+n_3a_3+n_4a_4+n_5a_5&=\frac{3}{4}n_1+\frac{1}{4}n_2+\frac{1}{2}n_4
\end{align*}
if $9\leq d\leq 11$, and
\begin{align*}
n_1a_1+n_2a_2+n_3a_3+n_4a_4+n_5a_5&=\frac{19}{24}n_1+\frac{1}{4}n_2+\frac{1}{2}n_4+\frac{1}{3}m\\
&=\frac{19}{24}n_1+\frac{1}{4}n_2+\frac{1}{2}n_4+\frac{1}{3}(d-2n_1-n_2-n_3-n_4-n_5)\\
&=\frac{1}{3}d+\frac{1}{8}n_1-\frac{1}{12}n_2-\frac{1}{3}n_3+\frac{1}{6}n_4-\frac{1}{3}n_5.
\end{align*}
if $d\geq 12$.

Now, we calculate the largest possible value of the charges sent by $v$ to its adjacent true small vertices in $\gx$. Note that we should only consider the case $d\geq 11$ by Lemma \ref{degre-sum}. Since no two true small vertices are adjacent in $G$, in each cluster of type 1 or 3 $v$ is adjacent to at most one true small vertex in $\gx$. This implies that $v$ is adjacent to at most $n_1+n_2+n_3+n_4+n_5-\frac{1}{2}(n_1+n_3)=\frac{1}{2}(n_1+n_3)+n_2+n_4+n_5$ true small vertices in $\gx$. Recall the definition of tri-neighbors at the beginning of this section. One can see that $v$ can be tri-neighbors of at most $n_3$ vertices. Therefore, $v$ sends at most $$\frac{1}{6}(n_1+n_3)+\frac{1}{3}(n_2+n_4+n_5)+\frac{1}{12}n_3$$ to its adjacent true small vertices in $\gx$ by R3. Note that
R2 cannot be applied to $v$ if $6\leq d\leq 12$, since the application of R2 implies $\Delta=r\geq 13$ by Lemma \ref{degre-sum},
and that $v$ may send $\frac{1}{2}$ to a common pot by R2 if $d\geq 13$.

We combine those lines of calculation. Let $\gamma_d$ be the largest possible value of the charges sent by $v$ if $d_G(v)=d$. We have
\begin{align*}
\gamma_8&=\frac{1}{2}n_1+\frac{1}{12}n_2+\frac{1}{6}n_4\\
\gamma_9=\gamma_{10}&=\frac{3}{4}n_1+\frac{1}{4}n_2+\frac{1}{2}n_4,\\
\gamma_{11}&=\frac{3}{4}n_1+\frac{1}{4}n_2+\frac{1}{2}n_4+\frac{1}{6}(n_1+n_3)+\frac{1}{3}(n_2+n_4+n_5)+\frac{1}{12}n_3\\
&=\frac{11}{12}n_1+\frac{7}{12}n_2+\frac{1}{4}n_3+\frac{5}{6}n_4+\frac{1}{3}n_5,\\
\gamma_{12}&=\frac{1}{3}d+\frac{1}{8}n_1-\frac{1}{12}n_2-\frac{1}{3}n_3+\frac{1}{6}n_4-\frac{1}{3}n_5+\frac{1}{6}(n_1+n_3)+\frac{1}{3}(n_2+n_4+n_5)+\frac{1}{12}n_3\\
&=4+\frac{7}{24}n_1+\frac{1}{4}n_2-\frac{1}{12}n_3+\frac{1}{2}n_4, \\and\\
\gamma_d&=\frac{1}{3}d+\frac{1}{8}n_1-\frac{1}{12}n_2-\frac{1}{3}n_3+\frac{1}{6}n_4-\frac{1}{3}n_5+\frac{1}{6}(n_1+n_3)+\frac{1}{3}(n_2+n_4+n_5)+\frac{1}{12}n_3+\frac{1}{2}\\
&=\frac{1}{3}d+\frac{7}{24}n_1+\frac{1}{4}n_2-\frac{1}{12}n_3+\frac{1}{2}n_4+\frac{1}{2}\\if~d\geq 13.
\end{align*}

For each $8\leq d\leq 12$, we consider the following program $\mathcal{P}_d$:
\begin{align*}
\max~&\gamma_d\\
{\rm s.t.~} &2n_1+2n_2+n_3+3n_4+n_5\leq d\\
& n_1,n_2,n_3,n_4,n_5,d\in \mathbb{Z}^+.
\end{align*}

Let $q_d$ be the optimal value of the program $\mathcal{P}_d$.

Since $\gamma_8\leq \frac{1}{4}(2n_1+2n_2+n_3+3n_4+n_5)\leq 2$, $q_8\leq 2$.

Since $\gamma_9\leq \frac{3}{8}(2n_1+2n_2+n_3+3n_4+n_5)-\frac{3}{8}(n_3+n_4+n_5)\leq 3$, $q_9\leq 3$. Note that if $2n_1+2n_2+n_3+3n_4+n_5=9$, then $n_3+n_4+n_5\geq 1$.

Since $\gamma_{10}\leq \frac{3}{8}(2n_1+2n_2+n_3+3n_4+n_5)\leq \frac{15}{4}$, $q_{10}\leq \frac{15}{4}$.

Since $\gamma_{11}\leq \frac{11}{24}(2n_1+2n_2+n_3+3n_4+n_5)-\frac{1}{8}(n_2+n_3+n_4+n_5)\leq \frac{59}{12}$, $q_{11}\leq \frac{59}{12}$. Note that if $2n_1+2n_2+n_3+3n_4+n_5=11$, then $n_2+n_3+n_4+n_5\geq 1$.

Since $\gamma_{12}\leq 4+\frac{1}{6}(2n_1+2n_2+n_3+3n_4+n_5)\leq 6$, $q_{12}\leq 6$.

Therefore, $c'(v)\geq d-6-q_d\geq 0$ for each $8\leq d\leq 12$.

If $d\geq 13$, then $2n_1+2n_2+n_3+3n_4+n_5\leq d$ implies $\gamma_d-(d-6)\leq \frac{1}{6}(2n_1+2n_2+n_3+3n_4+n_5)-\frac{2}{3}d+\frac{13}{2}\leq \frac{13-d}{2}\leq 0$.
Therefore, $c'(v)\geq d-6-\gamma_d\geq 0$ for $d\geq 13$.

\end{document}